\documentclass[12pt,a4paper]{article} %modify affilation?!!!! Federal State Budgetary Educational Institution of Higher Education "Saint-Petersburg State University?!!
\usepackage{amsmath,amssymb,amsthm,amsfonts,amscd,euscript,verbatim, t1enc, newlfont, graphicx, cancel}
\usepackage{hyperref}
\usepackage[all]{xy}
\providecommand{\keywords}[1]{\textbf{\textit{Key words and phrases }} #1}
\providecommand{\subjclass}[1]{\textbf{\textit{2020 Mathematics Subject Classification.}} #1}

\theoremstyle{definition}
\newtheorem{theo}{Theorem}[section]

\newtheorem{pr}[theo]{Proposition}

 \newtheorem{lem}[theo]{Lemma}
 \newtheorem{coro}[theo]{Corollary}
  
\theoremstyle{remark}
\newtheorem{rema}[theo]{Remark}

\theoremstyle{definition}
\newtheorem{defi}[theo]{Definition}

\numberwithin{equation}{section}

\newcommand\cu{\underline{C}}
\newcommand\du{\underline{D}}

\newcommand\au{\underline{A}}

\newcommand\bu{\underline{B}}
\newcommand\hu{\underline{H}}

\newcommand\obj{\operatorname{Obj}}

\newcommand\id{\operatorname{id}}

\newcommand\add{\operatorname{add}}

\DeclareMathOperator\kar{\operatorname{Kar}}
\DeclareMathOperator\wkar{\operatorname{wKar}}

\DeclareMathOperator\co{\operatorname{Cone}}

\newcommand\hw{{\underline{Hw}}}

\newcommand\wstu{w_{st}}
\newcommand\wstub{w_{st}^b}
\newcommand\hwstub{\hw_{st}^b}

\newcommand\z{{\mathbb{Z}}}

 \newcommand\lan{\langle}
\newcommand\ra{\rangle}

\newcommand\ns{\{0\}}

\newcommand\cp{\mathcal{P}}

\newcommand\opp{^{op}}

\begin{document}

\title
 {The hearts of weight structures are the weakly idempotent complete categories}%of compact objects?! wws?!
\author{Mikhail V. Bondarko, Sergei V. Vostokov
   \thanks{ %%!!!! The main results of the paper were  obtained under support of the Russian Science Foundation 
	This work was funded by the Russian Science Foundation under  grant no. 16-11-00200.}}\maketitle
\begin{abstract} 
In this note we %modify the existence of weight structures arguments to obtain
 prove that additive categories that occur as hearts of weight structures are precisely the {\it weakly idempotent complete}
 categories, that is, the categories where all split monomorphisms give direct sum decompositions. We also give several other conditions equivalent to weak idempotent completeness (some of them are completely new), and discuss weak idempotent completions of additive categories. \end{abstract}

 % quasi-coherent sheaves.}

%????\tableofcontents

 \section*{Introduction}
%This paper is dedicated to 
The goal of this note is to study additive categories that can occur as hearts of weight structures on triangulated categories.\footnote{Recall that weight structures are certain "cousins" of $t$-structures (see Remark \ref{rrts} below) that were introduced in \cite{bws} and \cite{konk}; in the latter paper they were called co-t-structures. Weight structures have several interesting applications to representation theory, motives, and algebraic topology; see \cite{bonspkar} for some references.} Actually, an  answer to this question can be extracted from Theorem 4.3.2(I,II) of \cite{bws}; yet the corresponding calculation of hearts does not contain all the detail. For this reason, in the current paper we study the corresponding {\it weakly idempotent complete} (additive) categories and {\it weak idempotent completions} in detail.  % (the history of this terminology is described in Remark \ref{rterm}).  
 Another notion important for this paper is the {\it weak retraction-closure} (of a subcategory; see Definition \ref{dwrcl}). %we give this definition). 

Let us briefly describe %some of the results
the contents of the paper. An additive category $\bu$  is said to be weakly idempotent complete if any $\bu$-split monomorphism gives a direct sum decomposition (see Definition \ref{dstable}(1)); obviously, any idempotent complete category is weakly idempotent complete. $\bu$ is said to be weakly retraction-closed in $\bu'\supset\bu$ if for a $\bu'$-isomorphism $Y\cong X\bigoplus Z$ the object $Z$ belongs to $\obj \bu$  whenever $X$ and $Y$ do. In \S\ref{swic} we prove that $\bu$ is weakly idempotent complete if and only if it is weakly retraction-closed in a (weakly) idempotent complete category $\bu'\supset \bu$. These conditions are equivalent to the existence of $\bu''\subset \bu$ and an idempotent complete $\bu'\supset \bu$ such that $\bu$ equals the corresponding weak retraction-closure of $\bu''$ in $\bu'$. Moreover, the weak retraction-closure of $\bu$ in the idempotent completion $\kar(\bu)$ gives a canonical weak idempotent completion $\wkar(\bu)$ of $\bu$, and we prove that the %(probably, well known)??
  universality of the $\kar$-construction also yields that of the  $\wkar$-one. Furthermore,  %we prove that 
	 $\bu$ is weakly idempotent complete if and only if any contractible bounded $\bu$-complex %(that is, a bounded complex homotopy equivalent to $0$) 
	splits (into a direct sum of isomorphisms; see Proposition \ref{pstable}(\ref{ise3})).

In \S\ref{swws} we recall some basics on weight structures. Recall that these are given by classes $\cu_{w\le 0}$ and $\cu_{w\ge 0}$ of objects of a triangulated category $\cu$; the heart $\hw$ of $w$ is the additive subcategory $\cu_{w\le 0}\cap \cu_{w\ge 0}$. 
 The aforementioned Theorem 4.3.2(I,II) of \cite{bws} (along with the somewhat stronger Corollary 2.1.2 of \cite{bonspkar}) gives an almost complete characterization of {\it bounded} weight structures. %"almost one-to-one" correspondence between {\it }
%Combining loc. cit. with Theorem \ref{tstable} one obtains that 
 Loc. cit. implies that any (additive) {\it connective} subcategory $\bu$ of $\cu$ gives a  canonical bounded weight structure $w$ on the smallest strictly full triangulated subcategory $\du$ of $\cu$ that contains $\bu$, and $\hw$ consists of $\du$-retracts of objects of $\bu$. Now, Theorem \ref{tstable} implies that $\hw$ is equivalent to $\wkar(\bu)$; thus $\hw$ is equivalent to $\bu$ whenever $\bu$ is weakly idempotent complete. Moreover, the results of \S\ref{swic} easily imply that  weakly idempotent complete categories are precisely the ones that occur as weakly retraction-closed subcategories of triangulated categories; they are also the categories equivalent to hearts of (bounded) weight structures. Furthermore, we prove that a full embedding $\bu\to \bu'$ induces an equivalence of $K^b(\bu) $ with $K^b(\bu')$ if and only if $\bu'$ is essentially a subcategory of $\wkar(\bu)$, and $K_0^{\add}(\bu)\cong K_0^{\add}(\bu')$ if this is the case.

 % and discuss their relationship to our results (along with their properties).
%The authors are deeply grateful to 
%In \S\ref{shearts} we prove that weakly idempotent complete categories are precisely the categories that occur as  hearts of weight structures; this is an easy consequence . Moreover, %, and prove some other statements related to this one.

\section{On additive categories and (weak) retraction-closures}%{Some categorical notation and definitions}
\label{snotata}

All categories and  functors  (including embedding ones) in this paper will be additive. 

\begin{itemize}

\item Given a category $C$ and  $X,Y\in\obj C$  we will write $C(X,Y)$ for  the set of morphisms from $X$ to $Y$ in $C$.

\item For categories $C'$ and $C$ we write $C'\subset C$ if $C'$ is a full %strict
subcategory of $C$.

\item Given a category $C$ and  $X,Y\in\obj C$, we say that $X$ is a {\it
retract} of $Y$ %(and $Y$ is {\it coretract} of $X$)
 if $\id_X$ can be %factorized as $X\stackrel{i}{\to} Y\stackrel{p}{\to}X$
 factored through $Y$.\footnote{Clearly,  if $C$ is triangulated % or abelian, %{del}
then $X$ is a retract of $Y$ if and only if $X$ is its direct summand.}\ 

%\item The symbol $\cu$ below will always denote some triangulated category. %, and $\bu$ will always be additive.

\item A %n additive (not necessarily additive) subcategory $\hu$ 
 class of objects $D$ in (an additive category) $\bu$ %an additive category $C$ 
%$\hu\subset C$ a subcategory $\hu$ is called
is said to be {\it retraction-closed} in $\bu$ if it contains all $\bu$-retracts of its elements. % in $C$.

\item  For any $(\bu,D)$ as above we will write %the full subcategory 
$\kar_{\bu}(D)$ %of %an additive category
 %$C$ whose objects
for the class of all $\bu$-retracts of %objects of 
 %(finite) direct sums of 
 elements of $D$. %objects %of a subcategory 
%$\hu$ in $C$ will be called the {\it Karoubi-closure} of $\hu$ in $C$; note that this subcategory is obviously additive and retraction-closed in $C$. 

\item We will say that %an additive category %$C$ 
 $\bu$ is {\it idempotent complete} if any idempotent endomorphism gives a direct sum decomposition in it; cf. Definition 1.2 of \cite{bashli}.

\item The {\it idempotent completion} $\kar(\bu)$ (no lower index) of %an additive category 
 $\bu$ is the category of ``formal images'' of idempotents in $\bu$.
Respectively, its objects are the pairs $(B,p)$ for $B\in \obj \bu,\ p\in \bu(B,B),\ p^2=p$, and the morphisms are given by the formula %(\ref{mthen}) below,
%\begin{equation}\label{mthen}
$$\kar(\bu)((X,p),(X',p'))=\{f\in \bu(X,X'):\ p'\circ f=f \circ p=f \}.$$ %\end{equation}
 The correspondence  $B\mapsto (B,\id_B)$ (for $B\in \obj \bu$) fully embeds $\bu$ into $\kar(\bu)$, and it is well known that $\kar(\bu)$ is essentially the smallest idempotent complete %additive 
 category containing $\bu$; see Proposition 1.3 of ibid.
\end{itemize}
%to fix notation???!

Now we will give  definitions that appear to be (more or less) new.

\begin{defi}\label{dwrcl}
Let $\bu'$ be an (additive) subcategory of $\bu$.

1. We will write $\wkar_{\bu}(\bu')$ for the full subcategory of $\bu$ whose objects are those  $Z\in \obj \bu$ %objects $Z$ of $\bu$ 
 such that there exist $X,Y\in \obj \bu'$ %and $Z\in \obj \bu'$ if
 with $X\bigoplus Z\cong Y$. We will call $\wkar_{\bu}(\bu')$ the {\it weak retraction-closure} of $\bu'$ in $\bu$. %  then $Z$ is an object of $\bu$ as well.

2. We will say that $\bu'$ is {\it weakly retraction-closed} in $\bu$ if $\wkar_{\bu}(\bu')=\bu$.
%+ essentially????
\end{defi}

Below we will need the following simple statements.

\begin{lem}\label{lwrcl}
Let $\bu'$ be a subcategory of $\bu$.

1. If $\bu'$ is retraction-closed in $\bu$ then it is also weakly retraction-closed in $\bu$.

%If $\bu'$ is a  subcategory of $\bu$ then 
2. $\wkar_{\bu}(\bu')$ is weakly retraction-closed in $\bu$.
\end{lem}
\begin{proof}
1. Obvious.

2. For an object $Z$ of $\bu$ and $X,Y\in\obj \wkar_{\bu}(\bu')$ such that $X\bigoplus Z\cong Y$ we should prove that $Z$ is an object of $\wkar_{\bu}(\bu')$ as well. Now we recall Definition \ref{dwrcl}(1) and choose $X_1,X_2,Y_1,Y_2\in \obj \bu'$ such that $X\bigoplus X_1\cong X_2$ and  $Y\bigoplus Y_1\cong Y_2$. Then $Z\bigoplus (X_2\oplus Y_1)\cong Y\bigoplus Y_1\bigoplus X_1\cong Y_2\bigoplus X_1 $. Since both $X_2  \bigoplus Y_1$   and  $Y_2\bigoplus X_1$ are objects of $\bu'$, we obtain the result. 
\end{proof}

\section{On weakly idempotent complete categories}\label{swic}

Let us give some more definitions. Throughout this paper $\bu$ will be an (additive) category.

\begin{defi}\label{dstable}
%Let $\bu$ be an (additive) category.
% We will say that a %weakly split; {neederex}; Thomason??
1. We will say that %an (additive) category
 $\bu$ is {\it weakly idempotent complete} if %for any $X,Y\in \obj \bu$ such that $\id_X$ factors through $Y$ there exists an object $Z$ of $\bu$ such that this factorization comes from an isomorphism $Y\cong X\bigoplus Z$.
  any split $\bu$-monomorphism $i:X\to Y$ (that is, $\id_X$ equals $p\circ i$ for some $p\in \bu(Y,X)$)  is isomorphic to the %obvious 
	 monomorphism $\id_X\bigoplus 0:X\to X\bigoplus Z$ for some object  $Z$ of $\bu$. 
	
	2. Assume that $\bu$ is  essentially small. Then the split Grothendieck group $K_0^{\add}(\bu)$  is the abelian group whose generators are %of the form $[X],\X\in\obj A$; 
the isomorphism classes of objects of $\bu$, and the relations are of the form $[B]=[A]+[C]$ for all   $A, B,C\in\obj \bu$ such that $B\cong A\bigoplus C$.  
\end{defi}

Now we prove that this definition is equivalent to several other ones.

\begin{pr}\label{pstable}
The following assumptions on %an additive category
  $\bu$ are equivalent.

\begin{enumerate}
\item\label{ise1} $\bu$ is weakly idempotent complete.

\item\label{ise2c} $\bu$ is  weakly retraction-closed in any (additive) category $\bu'$ containing $\bu$ as a strictly full subcategory.

%is a weakly retraction-closed subcategory (see Definition \ref{dwrcl}(2)) of an  idempotent complete category $\bu'$.

\item\label{ise2w} $\bu'$ is a weakly retraction-closed subcategory  of %a 
 some weakly  idempotent complete category $\bu'$.
%For  any  additive category $\bu'$ containing $\bu$ as a strictly full subcategory if $X,Y\in \obj \bu$, $Z\in \obj \bu'$, and $X\bigoplus Z\cong Y$ then $Z$ is an object of $\bu$ as well.

\item\label{iwic} The obvious embedding of $\bu$ into the category $\wkar(\bu)=\wkar_{\kar(\bu)}(\bu)$ (see Definition \ref{dwrcl}(1)) is an equivalence.

\item\label{ise2p} $\bu$ is equivalent to the category $\wkar(\bu'')$ for some (additive) category $\bu''$. 

\item\label{ise2} There exist additive categories $\bu''\subset \bu\subset \bu'$ such that $\bu'$ is idempotent complete and $\bu=\wkar_{\bu'}(\bu'')$.

\item\label{ise3} If a bounded $\bu$-complex is contractible (i.e., it is zero in $K^b(\bu)$) then it splits, that is, it %is $C(\bu)$-isomorphic to a complex of 
  has the form $\bigoplus \id_{N^i}[-i]$ for some $N^i\in \obj \bu$. %specify degrees? R: wrong for unbounded complexes?!
%its terms $M^i$ are of the form $N^i\bigoplus N^{i+1}$ for some $N^j\in \obj \bu$ and the boundary $d^i:M^i\to M^{i+1}$ equals $\id_{M^{i+1}}\bigoplus 0$. 
\end{enumerate}
\end{pr}
\begin{proof}
%By definition, if $\bu'$ is  idempotent complete 
%Our definitions clearly yield that 
 Obviously, condition \ref{ise1} implies condition \ref{ise2c}, and condition \ref{iwic} implies  condition \ref{ise2p}. Next, replacing $\bu$ by its isomorphism-closure in $\kar(\bu)$ we obtain that condition \ref{ise2c} implies condition \ref{iwic}. Moreover, if  $\bu$ is equivalent to the category $\wkar(\bu'')$ then we can replace  $\bu''$ and $\kar(\bu'')$ by  equivalent categories so that $\bu''\subset \bu\subset \bu'$, $\bu'$ is equivalent to $\kar(\bu'')$, and $\bu$ is a strict subcategory of $\bu'$. Hence condition  \ref{ise2p} implies condition \ref{ise2}.

 %If $\bu$ is weakly idempotent complete then for $\bu''=\bu$ the assumptions of condition \ref{ise2p} are fulfilled; hence this condition follows from condition \ref{ise1}.
Next, applying Lemma \ref{lwrcl}(2) %one easily obtains 
 we obtain that condition \ref{ise2p}  implies condition  \ref{ise2w}; note that %idempotent complete categories are 
 %the category $\kar(\bu'')$ is clearly
$\bu'$ is weakly idempotent complete since it is idempotent complete. 

Now assume that $\bu'$ is a weakly retraction-closed subcategory  of a weakly  idempotent complete category $\bu'$. We should prove that  
any  split $\bu$-monomorphism $i:X\to Y$ is isomorphic to the %obvious 
	 monomorphism $\id_X\bigoplus 0:X\to X\bigoplus Z$ for some object  $Z$ of $\bu$. Since $\bu'$ is weakly  idempotent complete, we obtain that $Z$ as desired exists in the category $\bu'\supset \bu$. Since $\bu$ is a weakly retraction-closed subcategory  of $\bu'$ and $X\bigoplus Z\cong Y$, we obtain $Y\in \obj \bu$; hence $\bu$ is weakly idempotent complete indeed.

Lastly we prove the equivalence of conditions  \ref{ise1} and  \ref{ise3}. If  %the composition of 
$p\circ i=\id_X$ for some $\bu$-morphisms $X\stackrel{i}{\to} Y\stackrel{p}{\to} X$ then %$i\circ p$ is ide 
the %chain of morphism 
 complex $$\dots\to 0\to X\stackrel{i}{\to} Y\stackrel{\id_Y-i\circ p}{\longrightarrow} Y \stackrel{p}{\to} X\to 0\to\dots$$ is easily seen to be split in $K(\kar \bu)$; hence it is zero in $K(\bu)$ as well. If it is also split in $K(\bu)$ then $i$ and $p$ come from a $\bu$-isomorphism $Y\cong X\bigoplus Z$; thus condition \ref{ise3} follows from condition \ref{ise1}. % as well.

Let us establish the converse implication by the induction on the essential length of a  complex $M=(M^i)$; that is, we look for the minimal $l\ge 0$ such that the terms $M^i$ are zero for $i<m$ and $i>n$, where $n-m=l$. Contractible complexes (over an arbitrary additive category) obviously splits if its essentiall length % $l$
  is at most  $1$. Now, assume that $M$ is contractible of length $l\ge 2$, and all contractible complexes of length less than $l$ split. Clearly, the contracting homotopy provides a factorization of $id_{M^m}$ through the boundary $d^m:M^m\to M^{m+1}$. Hence the complex $M$ is isomorphic to $\co(\id_{M^m})[-1-m]\bigoplus M'$, where $M'$ is of length $l-1$. Obviously, $M'$ is contractible as well and we conclude by applying the inductive assumption.
%Next, we take $\bu'=\kar(\bu)$ and immediately obtain that  condition \ref{ise3} implies condition \ref{ise1}.
\end{proof}
%stabilization? stable envelope?!!! combine??!

\begin{rema}\label{rstable}
\begin{enumerate}
\item\label{irse1} 
The notions of a weakly retraction-closed subcategory and of the weak retraction-closure are obviously self-dual.

Hence conditions \ref{ise2c}, \ref{iwic}, \ref{ise2p}, \ref{ise2} of our proposition are  self-dual as well; this is also true for condition \ref{ise3} (that is, these assumptions are fulfilled for $\bu$ if and only if  they are valid for $\bu\opp$).
% of our proposition is clearly self-dual  (that is, it is fulfilled for $\bu$ if and only if it is valid for $\bu\opp$). 
 Thus the notion of weak idempotent completeness is self-dual. %, 
Hence weak idempotent completions can also be characterized by the duals of %the %(remaining) 
 conditions \ref{ise1} and \ref{ise2w}  in Proposition \ref{pstable}. In particular,  we obtain Lemma 7.1 of \cite{buhler}. %the main result of \cite{freydsplit}. %(with respect to $\bu$)??! All of them are?!

\item\label{irwkar4}
We will call the category $\wkar(\bu)=\wkar_{\kar(\bu)}(\bu)$ the {\it weak idempotent completion} of $\bu$ following Remark 7.8 of \cite{buhler}. % (cf. Definition 4.3.1(3) of \cite{bws}).  
We will %now %give a proof of the claim made in loc. cit. (and extend it a little), and consequently 
 justify this terminology and also %justify 
 prove and extend the claim made in loc. cit. in Corollary \ref{cwidc} below.

\item\label{irse2p}
Let $R$ be an (associative unital) ring. Let us %specify 
 describe certain categories that fulfil the assumptions of Proposition \ref{pstable}(\ref{ise2}).

 Take
%take $\bu''\subset \bu'$ to  be the
 $\bu''$ to be the category  of free left finitely generated $R$-modules and $\bu'$ to be the category of all left $R$-modules. %, respectively. 
 Then   the corresponding  category $\bu\cong \wkar(\bu'')$ %as in condition \ref{ise2} 
  is just the category of  finitely generated %??? 
	stably free left $R$-modules.\footnote{The authors are deeply grateful to Vladimir Sosnilo for this nice observation.}

This example demonstrates that weakly idempotent complete categories do not have to be idempotent complete and gives a nice example of weak idempotent completions (along with weak retraction-closures). %; see part \ref{irwkar4} of this remark for the former term).

\item\label{irse3} The argument used in the proof of the implication (\ref{ise1})$\implies$(\ref{ise3}) easily implies that any bounded above or below contractible $\bu$-complex splits as well. %????? cf. Remark 1.10 of  of \cite{neederex} for a certain generalization of this statement.

%\item\label{irse4}
On the other hand,  %recall that %if $p$ is an idempotent endomorphism of $B\in \obj \bu$  then the $2$-periodic complex $\dots \to B\stackrel{\id_B-p}{\longrightarrow} B\stackrel{p}{\to} B \stackrel{\id_B-p}{\longrightarrow} B\to \dots$ does not split unless $p$ gives a splitting in $\bu$; thus
 Proposition 10.9 of \cite{buhler} %(cf. Remark 1.8 of \cite{neederex}) 
 says that arbitrary (unbounded) contractible $\bu$-complexes split if and only if $B$ is idempotent complete. 

These statements (along with our arguments above) are closely related to  Remark 1.12 of \cite{neederex}.
% (cf. Remark 1.8 of \cite{neederex}).
\end{enumerate}
\end{rema}

%Let us now discuss the functoriality of the 

\begin{coro}\label{cwidc} Let $F:\bu_1\to\bu_2$ be an additive functor.

1. Then %the obvious 
 there exists a natural "idempotent complete version" $\kar(F):\kar(\bu_1)\to \kar(\bu_2)$ that restricts to a functor  $\wkar(F):\wkar(\bu_1)\to \wkar(\bu_2)$.

2. Consequently, if $\bu_2$ is (weakly) idempotent complete then $F$ extends to an additive functor from $\kar(\bu_1)$ (resp. from $\wkar(\bu_1)$) into $\bu_2$.

3. Assume that $\bu$ is  essentially small.   Then $\kar(\bu)$ also is, and $\wkar(\bu)$ consists of those %objects of 
 $M\in \obj \kar(\bu)$ such that the class of  $M$ in % $[M]\in K_0(\kar(\bu))$ 
  $K_0(\kar(\bu))$ (see Definition \ref{dstable}(2)) belongs to the image of the obvious homomorphism $K_0^{\add}(\bu)\to K_0^{\add}(\kar(\bu))$.  %also true for intermediate subcategories??
%$\kar(F)$ (resp. 
\end{coro}
\begin{proof}
1. It is easily seen that %there exists 
 $F$ yields a canonical  additive functor $\kar(F)$ that sends  $(B,p)$ for $B\in \obj \bu_1,\ p\in \bu_1(B,B),\ p^2=p$ into $(F(B),f(p))$ indeed. 

Next, if $X\bigoplus Z\cong Y$ in $\kar(\bu_1)$ then $\kar(F)(X)\bigoplus \kar(F)(Z)\cong \kar(F)(Y)$. Thus if an object $Z$ of $\kar(\bu_1)$ belongs to $\wkar(\bu_1)$ then $\kar(F)(Z)$ belongs to $\wkar(\bu_2)$ indeed.

2. If $\bu_2$ is idempotent complete then it is equivalent to the category $\kar(\bu_2)$; hence one can modify $\kar(F)$ to obtain the extension in question.

Similarly, if  $\bu_2$ is weakly idempotent complete then it is equivalent to the category $\wkar(\bu_2)$ (see condition \ref{iwic} in Proposition \ref{pstable}); thus one can modify $\wkar(F)$ to obtain the result.

3. The essential smallness of  $\kar(\bu)$ obviously follows from that of $\bu$.

Next we note that %By 
the definition of $K_0(\kar(\bu))$ immediately implies the following: we have $[M]=[N_1]-[N_2]$ for some objects $N_i$ of $\bu$ (being more precise, here we consider the objects $(N_i,\id_{N_i})$ of $\kar(\bu)$)  whenever there exists $B\in \obj \kar(\bu)$ such that $M\bigoplus B \bigoplus N_2\cong N_1\bigoplus B$. Since $B$ is a retract of %some $B'\in \obj \bu$
 an object of $\bu$, this is equivalent to the existence of  $B'\in \obj \bu$ such that $M\bigoplus B'  \bigoplus N_2\cong N_1 \bigoplus N\bigoplus B'$. %The statement in question 
  Our assertion follows immediately.
\end{proof}

\begin{rema}\label{rterm}
Let us now relate the terminology in the current paper to that in earlier ones.

It appears  that the term "weakly idempotent complete" for a category $\bu$ was introduced in \cite[Definition 7.2]{buhler}. % whereas in 
 In \cite{freydsplit} (probably, this is where this notion was originally introduced) %the terminology
 it was said that {\it retracts have complements} (in $\bu$),\footnote{Recall that the main Proposition of ibid. says that %retracts have complements in $\bu$ if $\bu$ is 
weakly idempotent complete categories closed with respect to countable coproducts are idempotent complete.} whereas in  Definition 1.11 of \cite{neederex} % was used; cf. also   of \cite[Definition 1.11]{neederex} where $\bu$ of this sort 
 it was said that $\bu$ is {\it semi-saturated}. Most of the conditions in Proposition \ref{pstable} and Theorem \ref{tstable} were not mentioned in these papers.

Recall also that weak idempotent completions were called {\it small envelopes} in Definition 4.3.1(3) of \cite{bws} and {\it semi-saturations} in \S1.12.1 of of \cite{neederex}. 
\end{rema}

\section{Weight structures: short reminder}\label{swws}

Let us start from the definition of a weight structure (note however that the only axiom of weight structures that we will mention %below 
 explicitly in this text is the axiom (i)). The symbol $\cu$ %below 
  in this paper will always denote some triangulated category.

 It will be convenient for us to use the following notation below: %for $A,B\subset \obj \cu$ we 
for $D,E\subset \obj \cu$ we will write $D\perp E$ if $\cu(X,Y)=\ns$ for all $X\in D$ and $Y\in E$.

\begin{defi}\label{dwstr}

I. A couple $(\cu_{w\le 0},\cu_{w\ge 0})$ of classes  %\subset\obj 
 of objects of $\cu$ %(of {\it $w$-negative} and {\it $w$-positive} objects, respectively)
will be said to give a weight structure $w$ on %a triangulated category  
$\cu$ if %they  satisfy 
 the following conditions are fulfilled.

(i) $\cu_{w\le 0}$ and $\cu_{w\ge 0}$ are %additive and 
retraction-closed in $\cu$ (i.e., contain all $\cu$-retracts of their objects).

(ii) {\bf Semi-invariance with respect to translations.}

$\cu_{w\le 0}\subset \cu_{w\le 0}[1]$ and $\cu_{w\ge 0}[1]\subset
\cu_{w\ge 0}$.

(iii) {\bf Orthogonality.}

$\cu_{w\le 0}\perp \cu_{w\ge 0}[1]$.

(iv) {\bf Weight decompositions}.

 For any $M\in\obj \cu$ there
exists a distinguished triangle
$$L_wM\to M\to R_wM {\to} L_wM[1]$$
such that $L_wM\in \cu_{w\le 0} $ and $ R_wM\in \cu_{w\ge 0}[1]$.
\end{defi}

We will also need the following definitions.

\begin{defi}\label{dwso}
%Let $i,j\in \z$; a
Assume that a triangulated category $\cu$ is endowed with a weight structure $w$, $i\in \z$.

\begin{enumerate}
\item\label{idh}
%For a weight structure $w$ on $\cu$
 The full category $\hw\subset \cu$ whose objects are
$\cu_{w=0}=\cu_{w\ge 0}\cap \cu_{w\le 0}$ %and morphisms are $\hw(Z,T)=\cu(Z,T)$ for $Z,T\in \cu_{w=0}$,
 is called the {\it heart} of %the weight structure
$w$.

\item\label{id=i}
 $\cu_{w\ge i}$ (resp. $\cu_{w\le i}$, resp. $\cu_{w= i}$) will denote the class $\cu_{w\ge 0}[i]$ (resp. $\cu_{w\le 0}[i]$, resp. $\cu_{w= 0}[i]$).

\item\label{idbob} %We will call $\cup_{i\in \z} \cu_{w\ge i}$ (resp. $\cup_{i\in \z} \cu_{w\le i}$) the class of {\it $w$-bounded below} (resp., {\it $w$-bounded above}) objects of $\cu$. 
We will  say that $(\cu,w)$ is {\it  bounded} and $\cu$ is a {\it bounded weighted %????? 
category}  if  %$\cu^b=\cu$
$\obj \cu=\cup_{i\in \z} \cu_{w\ge i}=\cup_{i\in \z} \cu_{w\le i}$.

% we will write $\cu^b$ for the full subcategory of $\cu$ whose objects are bounded both above and below. %?????

%\item\label{id=i}  $\cu_{w\ge i}$ (resp. $\cu_{w\le i}$, resp. $\cu_{w= i}$) will denote the class $\cu_{w\ge 0}[i]$ (resp. $\cu_{w\le 0}[i]$, resp. $\cu_{w= 0}[i]$).

\item\label{idrest} %????
Let $\du$ be a full triangulated subcategory of $\cu$.

We will say that $w$ {\it restricts} to $\du$ whenever the couple $w_{\du}= (\cu_{w\le 0}\cap \obj \du,\ \cu_{w\ge 0}\cap \obj \du)$ is a weight structure on $\du$.

%??? Moreover, in this case we will also say that $w$ is an {\it extension} of $w_{\du}$.
%\item\label{id6}  We will say that a subcategory $\hu\subset \cu$  is {\it negative} (in $\cu$) if $\obj \hu\perp (\cup_{i>0}\obj (\hu[i]))$.

\item\label{idneg}
 We will say that the subcategory %$\hu\subset \cu$
 $\hu\subset \cu$  is {\it connective} (in $\cu$) if $\obj \hu\perp (\cup_{i>0}\obj (\hu[i]))$.\footnote{ In earlier texts of the first author connective subcategories were called {\it negative} ones. Moreover, in several papers (mostly, on representation theory and related matters) a connective subcategory satisfying certain additional assumptions was said to be {\it silting}; this notion generalizes the one of {\it tilting}.}

\item\label{idense} %We will  $\lan \hu\ra$ for  the smallest  full retraction-closed triangulated subcategory of $\cu$ containing $\hu$; we will  call  $\lan \hu\ra$  the triangulated subcategory {\it densely generated} by $\hu$ (in particular, in the case $\cu=\lan \hu \ra$).\footnote{Alternatively, $\lan \hu\ra$  can be called the thick subcategory of $\cu$ generated by $\hu$.}
The  smallest   strictly  full triangulated subcategory of $\cu$ containing $\hu$ will be called the subcategory {\it strongly generated} by  $\hu$ in $\cu$. 

\item\label{iextcl} We will say that a class $\cp\subset \obj \cu$ is {\it extension-closed} if $\cp$  contains $0$ and for any $\cu$-distinguished triangle $A \to C \to B \to A[1]$ the object $B$ belongs to $\cp$ whenever (both) $A$ and $C$ do.
\end{enumerate}
\end{defi}

\begin{rema}\label{rstws}

1. A  simple (and still  quite useful) example of a weight structure comes from the stupid filtration on the homotopy category of cohomological complexes $K(\bu)$ for an arbitrary additive  $\bu$; it can also be restricted to the subcategory $K^b(\bu)$ of bounded complexes (see Definition \ref{dwso}(\ref{idrest})). In this case $K(\bu)_{\wstu\le 0}$ (resp. $K(\bu)_{\wstu\ge 0}$) is the class of complexes that are homotopy equivalent to complexes  concentrated in degrees $\ge 0$ (resp. $\le 0$); see Remark 1.2.3(1) of \cite{bonspkar} for more detail. %This notation will be ????We will use this notation below. 

The heart of the weight structure $\wstu$ is the retraction-closure  of $\bu$  in  $K(\bu)$; hence it is equivalent to $\kar(\bu)$ (since both $K^{-}(\bu)$ and $K^+(\bu)$ are idempotent complete).

  %(or in $K(\bu)$, respectively). 
The restriction of  $\wstu$ to $K^b(\bu)$ will be denoted by $\wstub$; %\footnote{Note that there can exist at most one} Proposition  1.2.5 of {bspure}
 % its heart $\hwstub$ will be discussed
 in Theorem \ref{tstable} %(\ref{ise6})
  below we will demonstrate that its heart $\hwstub$ is equivalent to $\wkar(\bu)$.

%\begin{enumerate}

%\item\label{irconv}
2. In %the current paper 
 this note we use the ``homological convention'' for weight structures. This is the convention used by several papers of the first author  (including \cite{bonspkar} and \cite{bwcp}).  However, in \cite{bws} the so-called cohomological convention was used; %Now, in the homological 
 in this convention the functor $[1]$ "shifts weights" by $-1$. This is one of the reasons for us not to cite ibid. below; another one is that the exposition of  the theory of weight complexes (that we will apply in the proof of Theorem \ref{tstable}) in \S3 of \cite{bws} is rather inaccurate. %; note in contrast that (in the cohomological convention for $t$-structures that originates from \cite{bbd} and was used in \cite{neesat}) $[1]$ "shifts $t$-degrees" by $-1$.  
%\end{enumerate}
\end{rema}

Let us now recall the relation of connective subcategories to weight structures.

\begin{pr}\label{pbw}
Let $\cu$ be a triangulated category.

I. Assume that $w$ is  a weight structure on $\cu$.

1.  Then  the classes $\cu_{w\le 0}$, $\cu_{w\ge 0}$, and $\cu_{w=0}$ are  extension-closed; consequently, they are additive. 

2.  Let  $v$ be another weight structure for $\cu$; suppose that  $\cu_{w\le 0}\subset \cu_{v\le 0}$ and $\cu_{w\ge 0}\subset \cu_{g\le 0}$. Then $w=v$ (i.e., the inclusions are equalities).

II. Under the assumptions of Definition \ref{dwso}(\ref{idneg}) there exists a unique weight structure $w_{\bu}$ on the category $\du=\lan \bu \ra$ whose heart contains $\bu$. 
Moreover, this weight structure is bounded and  $\du_{w_{\bu}=0}=\kar_{\cu}(\obj \bu)$. %????, and $\du_{w_{\bu}\ge 0}$ (resp. $\du_{w_{\bu}\le 0}$) is the smallest retraction-closed class of objects of $\cu$ that contains $\obj\bu[i]$ for all $i\ge 0$ (resp. $i\le 0$) and is closed with respect to extensions. %define???????  %, and is 
%also  extension-closed and retraction-closed in $\cu$.
 %\end{enumerate}
\end{pr}
\begin{proof}
I.1. See Proposition 1.2.4(3) and Remark 1.2.3(4) of \cite{bonspkar}.

2. This is Proposition 1.2.4(7) of loc. cit.

II. Immediate from %????  %This is (most of) 
 Corollary 2.1.2 of ibid. %\cite{bonspkar}.
\end{proof}

\section{On hearts of weight structures}\label{shearts}

\begin{theo}\label{tstable}
The following assumptions on (an additive category) $\bu$ are equivalent as well.

\begin{enumerate}
\item\label{iwse1} $\bu$ is weakly idempotent complete.

\item\label{iwset} There exists a  triangulated category $\cu$ such that $\bu$ is its weakly retraction-closed subcategory. % of some triangulated category $\cu$. % weakly idempotent complete.

%\item\label{itri} There exists a triangulated category $\cu\supset \bu$ and for any $C\in \obj \cu$ such that there exist distinguished triangles $B_1\to B_2\to C\to B_1[1]$ and  $ B_3[-1]\to C\to B_4\to B_3$ with $B_i\in \obj \bu$ we have $C\in \obj \bu$. Extension-closed?! Remark?! Nafig???!

\item\label{ise4} $\bu$ is equivalent to the heart of a weight structure.

\item\label{ise5} $\bu$ is equivalent to the heart of a bounded weight structure.

\item\label{ise6} $\bu$ is equivalent to the heart $\hwstub$ of the weight structure $\wstub$ on the category $K^b(\bu)$ (see Remark \ref{rstws}(1)).  

\item\label{ise7} For any category $\bu''$ such that the embedding $\bu''\to \wkar(\bu'')$ factors through a fully faithful functor $\bu\to \wkar(\bu'')$ %$\bu''\subset \bu$  this embedding yields an equivalence $\wkar(\bu'')\to \bu$,
 and $\bu''$  is connective in a triangulated category $\cu$ (see Definition \ref{dwso}(\ref{idneg})), there exists a unique weight structure $w$ on the triangulated subcategory $\du$ of $\cu$ strongly generated by $\bu''$ such that the heart $\hw$ is naturally equivalent to $\bu$ (that is, $\hw$ contains $\bu''$ and the embedding $\bu''\to \hw$ factors through an equivalence of $\bu$ with $\hw$).
\end{enumerate}
\end{theo}
\begin{proof}
Clearly, condition  \ref{ise6} implies condition   \ref{ise5},  and  \ref{ise5} implies \ref{ise4}.  Next,  %assume that $\bu''$ is as in 
  we can take $\bu''=\bu$ in condition \ref{ise7}. Since $\bu$ is connective in the category $\cu=\du=K^b(\bu)$ and strongly generates it, %applying Proposition \ref{pexw}(1) 
 we obtain that condition \ref{ise7} implies condition  \ref{ise6}.

Now, axiom (i) of Definition \ref{dwstr} implies that $\hw$ is retraction-closed in $\cu$ (note that it is an additive subcategory by Proposition \ref{pbw}(I.1)). Thus condition \ref{ise4} implies condition \ref{iwset}.

Furthermore, any triangulated category is easily seen to be weakly idempotent complete since for $X$ and $Y$ as in Definiton \ref{dstable}(1) we have $Y\cong X\bigoplus \co(X\to Y)$. Thus condition  \ref{iwset} implies that  $\bu$ is weakly idempotent complete (i.e., that condition \ref{iwse1} is fulfilled); see Proposition \ref{pstable}(\ref{ise2w}).

Thus it remains to verify that  any weakly idempotent complete category $\bu$ fulfils condition \ref{ise7}.  The existence and the uniqueness of a weight structure $w$ on $\du$ %in question is 
 such that $\bu''\subset \hw$ %is given by (a more general) 
  follows immediately from  Proposition \ref{pbw}(II); we also obtain  the existence of a fully faithful functor $\hw\to \kar(\bu'')$, whereas the latter category is clearly equivalent to $\kar(\bu)$. %Corollary 2.1.2 of \cite{bonspkar}. %;  %Moreover, loc. cit. says we also obtain that $\hw=\kar_{\cu}(\bu'')$. 
	%Thus it suffices to
	Moreover,  $\hw$ is weakly idempotent complete (recall that we have just proved that our condition \ref{ise4} implies condition  \ref{iwse1}); hence Corollary \ref{cwidc} implies that the embedding $\bu\to \bu''$ factors through a full embedding of $\bu$ into $\hw$.

	%Thus 
	 Since $\bu$ is weakly idempotent complete, %??(see  Proposition \ref{pstable}(\ref{iwic}))  
	 it remains to verify that for any $M\in \du_{w=0}$ there exist objects $X$ and $Y$ of $\bu$ such that $M\bigoplus X\cong Y$. We will deduce this statement from the existence of splittings of contractible complexes in %$K^b(\bu)\supset K^b(\bu'')$ 
	$K^b(\hw)$; % (see Proposition \ref{pstable}(\ref{ise3}); note that we have already proved that $\hw$ is weakly idempotent complete). For
	 for this purpose we invoke the theory of  {\it (weak) weight complex functors} as provided by Proposition 1.3.4 of \cite{bwcp}.
	
	Part 6 of  loc. cit. associates to $M$ its weight complex $t(M)\in \obj K(\hw)$. Parts 4 and 10 of loc. cit. imply that $t(M)\cong M$ (in the homotopy category $K(\hw)$). On the other hand, parts 4 and 9 easily yield that $t(M)$ is homotopy equivalent to a complex $N\in \obj K^b(\bu'')\subset K(\hw)$. Hence there exists a $K(\hw)$-morphism $f:M\to N$ such that $\co(f)$ is contractible. Since $\co(f)\in \obj K^b(\hw)$ and  we have already proved that $\hw$ is weakly idempotent complete, we obtain that  $\co(f)$ splits in $K^b(\hw)$. Now, if $N^i\in \obj \bu''$ are the terms of $N$, then this splitting yields $M\bigoplus_{j\in \z}N^{2i-1}\cong  \bigoplus_{j\in \z}N^{2i}$. % (cf. the proof of \cite[Proposition 1]{rose}). 
	 This concludes the proof.
	\end{proof}

\begin{rema}\label{rrts} Weight structures are well known to be closely related to $t$-structures (as introduced in  \S1.3 of \cite{bbd}). However, the properties of weight structures are significantly distinct from that of $t$-structures. Recall in particular that the hearts of $t$-structures are precisely the abelian categories.   Hence there are plenty of additive categories that are hearts of some weight structures and cannot occur as hearts of $t$-structures; cf. Remark \ref{rstable}(\ref{irse2p}).
\end{rema}

The following statement gives one more characterization of weak idempotent completions as well as certain Grothendieck group isomorphisms.

\begin{coro}\label{cequivkb}
1. If $\bu\subset \bu'$ then the corresponding embedding $K^b(\bu)\to K^b(\bu')$ is an equivalence if and only if the embedding $\bu\to \wkar(\bu)$ factors through  a fully faithful functor from $\bu'$ into $\wkar(\bu)$.

2. Consequently, if $\bu'$ is essentially small and $\bu\subset \bu' \subset \wkar(\bu)$ then  %we have 
the obvious homomorphism $K_0^{\add}(\bu)\to K_0^{\add}(\bu')$   (see Definition \ref{dstable}(2)) is bijective.
\end{coro}
\begin{proof} 1. Assume that the embedding $K^b(\bu)\to K^b(\bu')$ is an equivalence. Then we can assume that the stupid weight structure on  $K^b(\bu')$ (see Remark \ref{rstws}(1)) gives a weight structure $v$ on $\cu= K^b(\bu)$. Since the classes $ \cu_{v\le 0}$ and $\cu_{v\ge  0}$ are closed with respect to isomorphisms (see the axiom (i) in Definition \ref{dwstr}), we obtain $\cu_{\wstub\le 0}\subset \cu_{v\le 0}$ and $\cu_{\wstub\ge 0}\subset \cu_{v\ge 0}$. Thus $\wstub=v$ according to  Proposition \ref{pbw}(I.2). Since $\hwstub$ is equivalent to $\wkar(\bu)$ (see condition \ref{ise6} in Theorem \ref{tstable}), %$\hv\subset \hwstub$ $\hwstub\subset \hv$; since there exists a fully faithful functor from $\bu'$ into $\hv$,
  this clearly gives  a fully faithful functor from $\bu'$ into $\wkar(\bu)$. %the factorization of the embedding  

Now let us prove the converse implication. 
We should prove that $\du= K^b(\bu')$, where $\du$ is the closure of $K^b(\bu)$ in $ K^b(\bu')$ with respect to isomorphisms, if  $\bu'$  is equivalent to a subcategory of $\wkar(\bu)$. Now, $\du$ is a strictly full subcategory of $ K^b(\bu')$ that essentially contains %actually literally true??
 $\wkar(\bu)$ (see condition \ref{ise7} in Theorem \ref{tstable}). Since $ K^b(\bu')$ is clearly strongly generated (see Definition \ref{dwso}(\ref{idense})) by $\bu'$, we easily obtain the equality in question.

2. %Since the embedding 
 According to assertion 1, the embedding $K^b(\bu)\to K^b(\bu')$ is an equivalence in this case. Thus it suffices to recall that for any essentially small (additive) category $\au$ the group $K_0^{\add}(\au)$ can be computed as a certain triangulated Grothendieck group of the category $K^b(\au)$; see  Definition 2 and Theorem 1 of \cite{rose}. 
\end{proof}

\begin{rema}\label{requivk} 1. One can also prove that a category $\bu'\supset \bu$  is essentially a subcategory of $\kar(\bu)$ if and only if $K(\bu)\cong K(\bu')$; this is also equivalent to  $K^+(\bu)\cong K^+(\bu')$ and $K^-(\bu)\cong K^-(\bu')$. To prove the "if" implications here one can apply stupid weight structure arguments similar to the one %that we have just used
 in the proof of Corollary \ref{cequivkb}, and one can use  a reasoning similar to a one in Remark 3.3.2(2) of \cite{bwcp} to obtain the converse implications.

2. This observation along with Proposition \ref{pstable}(\ref{ise3}), Proposition 10.9 of \cite{buhler} (see Remark \ref{rstable}(\ref{irse3}), and Remark \ref{rstws} justifies the following vague claim: $\kar(\bu)$ is the "extension" of $\bu$ corresponding to unbounded $\bu$-complexes, wheras $\wkar(\bu)$ "corresponds to" $K^b(\bu)$.

3. One can also prove Corollary \ref{cequivkb}(2) more explicitly; cf. Corollary \ref{cwidc}(3).
\end{rema}

\end{document}